\documentclass[11pt, reqno]{amsart}
\usepackage{amssymb, amsfonts, euscript}

\newcommand{\ra}{\rangle}
\newcommand{\la}{\langle}

\newcommand{\until}{\upharpoonright}
\newcommand{\zj}{\emptyset'}

\theoremstyle{plain}
\newtheorem{theorem}{Theorem}
\newtheorem{lemma}[theorem]{Lemma}
\newtheorem{corollary}[theorem]{Corollary}

\theoremstyle{definition}
\newtheorem{definition}[theorem]{Definition}

\theoremstyle{remark}

\numberwithin{equation}{section}

\begin{document}

\title{Cupping with random sets}
\date{June 5, 2012}

\author{Adam R. Day}
\address{Department of Mathematics\\
University of California, Berkeley\\
Berkeley, CA 94720-3840 USA}
\email{adam.day@math.berkeley.edu}

\author{Joseph S.~Miller}
\address{Department of Mathematics\\
University of Wisconsin\\
Madison, WI 53706-1388, USA}
\email{jmiller@math.wisc.edu}

\thanks{The first author was supported by a Miller Research Fellowship in the Department of Mathematics at the University of California, Berkeley. The second author was supported by the National Science Foundation under grant DMS-1001847.}

\makeatletter
\@namedef{subjclassname@2010}{\textup{2010} Mathematics Subject Classification}
\makeatother


\begin{abstract}
We prove that a set is $K$-trivial if and only if it is not weakly ML-cuppable. Further, we show that a set below zero jump is $K$-trivial if and only if it is not ML-cuppable. These results settle a question of Ku\v{c}era, who introduced both cuppability notions.
\end{abstract}

\maketitle

\section{Introduction}

The study of algorithmic randomness has uncovered a
remarkable relationship between sets that are highly random and sets that have no random content at all.
In this paper, we strengthen this relationship by looking at the computational power of a set joined with an incomplete random set. The sets with no random content are known as the $K$-trivials.
We will prove that the $K$-trivial sets are precisely those that cannot be joined above $\zj$ with an incomplete random. We will also show that all sets below $\zj$ that are not $K$-trivial can be joined to $\zj$ with a low random.

The question as to whether the $K$-trivial sets could be characterized in this manner was originally posed by Ku\v{c}era \cite{Nies_2007}. 
This question appears in Miller and Nies's paper \textit{Randomness and computability: open questions} as well as Downey and Hirschfeldt's recent monograph \textit{Algorithmic Randomness and Complexity} \cite{Down_Hirs_2010,Mill_Nies_2006}.

We will assume familiarity with the basic theory of algorithmic randomness. For an introduction to this topic, see the recent books by Downey and Hirschfeldt and by Nies \cite{Down_Hirs_2010, Nies_2009}. We adopt the notational conventions used by Downey and Hirschfeldt. Recall that a set $A$ is \textit{$K$-trivial} if for some constant $c$, for all $n$ we have that $K(A\until n) \le K(n) +c$ (where $K(n)$ is defined to be $K(1^n)$).
 Hence, a $K$-trivial set is indistinguishable from a computable set in terms of $K$ complexity. The existence of non-computable $K$-trivial sets was first established by Solovay in an unpublished manuscript \cite{Down_Hirs_2010}. Later, Zambella constructed a non-computable $K$-trivial c.e.\ set \cite{Zamb_1990}.

A set $R$ is \textit{Martin-L\"of random} if there is a constant $c$ such that for all $n$,
$K(R\until n) > n-c$. The definition of a Martin-L\"of random set can be relativized to any oracle $A$. We call a set $A$ \textit{low for Martin-L\"of randomness} if every Martin-L\"of random set is also 
Martin-L\"of random relative to $A$.

The following theorem, a combination of results by Nies and by Hirschfeldt, Nies and Stephan, shows the relationship between Martin-L\"of randomness and $K$-triviality \cite{Hirs_Nies_Step_2006, Nies_2005}.

\begin{theorem}[Nies; Hirschfeldt, Nies and Stephan]\label{thm:K-trivial}
The following are equivalent:
\begin{enumerate}
\item $A$ is $K$-trivial.
\item $A$ is low for Martin-L\"of randomness.
\item There exists $X \ge_T A$ such that $X$ is Martin-L\"of random relative to~$A$.
\end{enumerate}
\end{theorem}

Posner and Robinson proved that any non-computable set that is Turing below $\zj$ can be cupped to $\zj$ with a 1-generic set \cite{Posn_Robi_1981}. This is a fundamental result in computability theory. 
In 2004, Ku\v{c}era asked which sets below $\zj$ can be cupped to $\zj$ with an incomplete Martin-L\"of random \cite{Nies_2007}. In other words, does the Posner-Robinson theorem hold if we replace Baire category with Lebesgue measure, and if not, for which sets does it fail? This question motivated the following definition.
\begin{definition}
A set $A$ is \textit{weakly ML-cuppable} if there is an incomplete Martin-L\"of random set $X$ such that $A\oplus X \ge_T \zj$. If one can choose $X <_T \zj$, then $A$ is \textit{ML-cuppable}.
\end{definition}

Ku\v{c}era suggested, quite correctly, that these notions might characterize non-$K$-triviality. Nies gave a partial answer to Ku\v{c}era's question by showing that there exists a non-computable $K$-trivial c.e.\ set that is not weakly ML-cuppable~\cite{Nies_2007}.

\section{Proving Non-ML-cuppability}

\begin{theorem} No $K$-trivial set is weakly ML-cuppable.
\label{thm: notcuppable}
\end{theorem}

The proof of Theorem~\ref{thm: notcuppable}
builds on work of Franklin and Ng, and Bienvenu, H\"olzl, Miller and Nies. Franklin and Ng characterized the incomplete Martin-L\"of random sets in terms of tests consisting of differences of $\Sigma^0_1$ classes \cite{Fran_Ng_2011}. 
Recently, Bienvenu, H\"olzl, Miller and Nies showed that the incomplete Martin-L\"of random sets are exactly those Martin-L\"of random sets for which a particular density property fails \cite{BiHoMiNi:12}. Let $P$ be a measurable set and $\tau$ a finite string. We write $d(\tau, P)$ for $\mu([\tau] \cap P) \cdot 2^{|\tau|}$ where $\mu$ is the Lebesgue measure. This is the density of $P$ above $\tau$.
We say that $X$ has lower density zero in $P$ if 
for all $\delta > 0$ there exists an $n$ such 
$d( [X\until n], P) < \delta$.

\begin{theorem}[Bienvenu, H\"olzl, Miller, Nies]
\label{thm: zero density} If $X$ is a Martin-L\"of random set, then $X$ is complete if and only if there exists a $\Pi^0_1$ class $P$ such that $X\in P$ and $X$ has lower density zero in $P$.
\end{theorem}

The utility of Theorem~\ref{thm: zero density} is that if a Martin-L\"of random set is complete, then there is a single $\Pi^0_1$ class that witnesses this fact. The following corollary is a direct relativization of one direction of this theorem.

\begin{corollary}
\label{cor: relative zero density}
Fix $A$. If $X$ is a set Martin-L\"of random relative to $A$ and $A\oplus X \ge_T A'$, then there exists a $\Pi^0_1(A)$ class $P$ such that $X \in P$ and $X$ has lower density zero in $P$.
\end{corollary}

We need one more result in order to prove Theorem~\ref{thm: notcuppable}. Let $S$ be a set of finite strings. We call $S$ \textit{bounded} if 
\[\sum_{\sigma \in S} 2^{-|\sigma|} < \infty.\]
The following result is implicit in Miller, Kjos-Hanssen and Solomon \cite{Kjos_Mill_Reed_toappear}, but it was first explicitly stated, in an even stronger form, by Simpson \cite[Lemma~5.11]{Simp_2007b}.

\begin{theorem}
\label{thm: covers}
Let $A$ be a $K$-trivial set and let $W_A$ be a bounded set of strings that is c.e.\ in \ $A$, then there exists a bounded c.e.\ set of strings $W$ such that $W_A \subseteq W$.
\end{theorem}

\begin{proof}[Proof of Theorem~\ref{thm: notcuppable}]
Let $A$ be a $K$-trivial set. Let $R$ be a Martin-L\"of random set such that $A\oplus R \ge_T \emptyset'$. Nies showed that all $K$-trivial sets are low, and thus $A\oplus R \ge_T A'$ \cite{Nies_2005}. Further, because all $K$-trivial sets are low for Martin-L\"of randomness, $R$ is Martin-L\"of random relative to $A$. Hence from Corollary~\ref{cor: relative zero density}, there exists a $\Pi^0_1(A)$ class $P_A$ such that $R\in P_A$ and $R$ has lower density zero in $P_A$. 

Let $W_A$ be an $A$-c.e.\ prefix-free set such that $P_A= \{X \in 2^\omega \colon (\forall \sigma \in W_A)\;\sigma \not \prec X\}$. Any prefix-free set has bounded weight. Hence by Theorem~\ref{thm: covers}, there is a c.e.\ set $W$ such that $W_A \subseteq W$ and $W$ has bounded weight. 

The fact that $W$ has bounded weight means that $W$ is a Solovay test. This implies that there are finitely many (perhaps zero) initial segments of $R$ in $W$. No initial segment of $R$ is in $W_A$, so we may remove them from $W$, preserving the fact that it is a bounded weight c.e.\ superset of $W_A$. Let $P= \{X \in 2^\omega \colon (\forall \sigma \in W)\;\sigma \not \prec X\}$. Observe that $R \in P$ and $P\subseteq P_A$.
As $R$ has lower density zero in $P_A$, it follows immediately that $R$ has lower density zero in $P$. So by Theorem~\ref{thm: zero density}, $R$ is complete. 
\end{proof}

\section{Constructing ML-cupping partners}

We will now show that for any set $A$ that is not $K$-trivial, there is an incomplete Martin-L\"of random set $R$ such that $A \oplus R \ge_T \zj$. Further, we will show that if $A$ is computable from $\zj$ then such an $R$ can found that is low. This means that any set that is not $K$-trivial is weakly ML-cuppable, and any set below $\zj$ that is not $K$-trivial is ML-cuppable.

The previous results in this direction have used Martin-L\"of random sets that are also Martin-L\"of random relative to $A$. The following was noted in Nies \cite{Nies_2007}. Assume that $A\leq_T\zj$ is not $K$-trivial. Let $\Omega^A$ be Chaitin's $\Omega$ relativized to $A$ (see \cite{DHMN_2005}). It follows from Theorem~\ref{thm:K-trivial} that $\Omega^A\ngeq_T A$, otherwise $A$ would be $K$-trivial. Hence $\Omega^A\ngeq_T \zj$. Relativizing the fact that $\Omega\equiv_T\zj$ gives us $A\oplus \Omega^A \equiv_T A'\geq_T\zj$, so $A$ is weakly ML-cuppable. Further, if $A$ is low, then $\Omega^A\leq_T A'\leq_T\zj$, so $A$ is ML-cuppable.

Our approach to finding a cupping partner for $A$ is different. We will construct a Martin-L\"of random set $R$ that is \emph{not} Martin-L\"of random relative to $A$. We will control the places where $R$ enters the levels of the universal Martin-L\"of test relative to $A$, and use this to compute a function from $A\oplus R$ that dominates the settling time function of $\zj$.

The set $R$ will be constructed using a $\zj$ oracle. Control over $R$ will be maintained by keeping $R$ inside a sequence of $\Pi^0_1$ classes. During the construction, we need to keep $R$ inside a $\Pi^0_1$ class while removing it from some $\Pi^0_1(A)$ class that contains only $A$-random sets (sets Martin-L\"of random relative to $A$). To achieve this, we use the fact that $A$ is not a $K$-trivial set, hence not low for Martin-L\"of randomness.

The following lemma is an immediate corollary of a theorem of Ku\v{c}era, who proved that for any Martin-L\"of random set $X$ and any $\Pi^0_1$ class $P$ of positive measure, there is a tail of $X$ in $P$ \cite{Kuce_1985}.
\begin{lemma}
Let $P$ be any $\Pi^0_1$ class of positive measure and let $Q$ be any $\Pi^0_1(A)$ class that only contains $A$-random sets. If $P \subseteq Q$ then $A$ is $K$-trivial.
\end{lemma}
\begin{proof}
If $X$ is a Martin-L\"of random set, then some tail of $X$ is in $P$. This implies that some tail of $X$ is in $Q$, and hence that $X$ is $A$-random.
\end{proof}

\begin{lemma}
\label{lem: extension}
Let $P$ be a $\Pi^0_1$ class, $\tau \in 2^{<\omega}$, and $A$ a set that is not $K$-trivial. Let $Q$ be any $\Pi^0_1(A)$ class that only contains $A$-random sets. 
If $d(\tau, P) \ge \delta$, then 
there exists $\rho \succ \tau$ such that
\begin{enumerate}
\item $\rho \not \in Q$. 
\item $d(\rho, P) \ge \delta/2$.
\item For all $\sigma\prec \rho$, there is an $s\in\omega$ with $\sigma\in Q[s]$ and $\rho\notin Q[s]$.
\end{enumerate}
\end{lemma}
\begin{proof}
Define the following $\Pi^0_1$ class:
\[\hat{P} = \{ X \in 2^\omega \colon X \succeq \tau \wedge (\forall \sigma)(\tau \preceq \sigma \prec X \rightarrow d(\sigma, P) \ge \delta/2)\}.\] 
If there is an $X\in \hat{P}\setminus Q$, then we can take $\rho\prec X$ to be the shortest extension of $\tau$ that is not in $Q$. So assume that $\hat{P} \subseteq Q$. If $\hat{P}$ has positive measure, we can apply the previous lemma to obtain a contradiction to the fact that $A$ is not $K$-trivial.

Let $S$ be a prefix-free set of strings that defines the complement of $\hat{P}$ above $\tau$. 
If $\sigma \in S$, then the measure of the complement of $P$ above $\sigma$ is at least $2^{-|\sigma|}(1- \delta/2)$. Hence,
\[\mu S \cdot (1-\delta /2) \le \mu ([\tau] \setminus P) \le (1-\delta)\cdot 2^{-|\tau|}.\]
This implies that the measure of $S$ is strictly less than $2^{-|\tau|}$, so $\hat{P}$ has positive measure.
\end{proof}

\begin{theorem}
\label{thm: cupping}
Let $A <_T \zj$ be a set that is not $K$-trivial. Let $D \le_T \zj$ be a non-computable set. There exists a Martin-L\"of random set $R$ such that
\begin{enumerate}
\item $R \not \ge_T D$.
\item $R$ is low.
\item $A \oplus R \equiv_T \zj$.
\end{enumerate}
 \end{theorem}
\begin{proof}
We will construct $R$ using a $\zj$ oracle. We will build a sequence $\{( \tau_s, P_s)\}_{s\in\omega}$ where $\tau_s$ is a finite string and $P_s$ is a $\Pi^0_1$ class. For all $s$, we will ensure that $\tau_{s+1} \succeq \tau_s$ and $P_{s+1} \subseteq P_s$. We will take $R = \bigcup_s \tau_s$, and further ensure that 
$R\in \bigcap_s P_s$. 
For all $n$, let $Q_n(A)$ be the complement of the $n$th level of the universal Martin-L\"of test relative to~$A$.

The idea behind this proof is to construct an $R$ that is not Martin-L\"of random relative to $A$. We can then use the stage that $R$ leaves $Q_{l(s)}(A)$ ($l$ will be a function computable in $A\oplus R$) to compute a function that dominates the settling time function of $\zj$. Let $m$ be the settling time function for $\zj$. The requirement that $R \not \ge_T D$ will be achieved at odd stages of the construction and the requirement that $R$ is low will be achieved at even stages in the construction. In particular we will ensure that:
\begin{enumerate}
\item If $s$ is odd and $e= (s-1)/2$, then $X \in [\tau_{s}] \cap P_{s} \rightarrow \Gamma_{e}^X \ne D$.
\item If $s>0$ is even and $e=s/2$, then either $( X \in [\tau_s] \cap P_{s} \rightarrow \Gamma_{e}^X(e) \uparrow)$ or 
$( X \in [\tau_{s}] \cap P_{s} \rightarrow \Gamma_e^X(e) \downarrow)$.
\end{enumerate}

Define $f:\omega \rightarrow \omega$ by $f(s) =2^{-4s-1}$. The function $f$ will be a lower bound for $d(\tau_s, P_s)$.
At stage $0$, let $\tau_0 =\lambda$ and let $P_0$ be the complement of the first level of the universal Martin-L\"of test.

At stage $s+1$, assume that $d(\tau_s, P_s) \ge f(s)$.
This clearly holds for the case that $s=0$.
Define $l(s) = |\tau_{s}| + 4s+2$. Observe that:
\begin{align*}
d(\tau_{s}, P_{s} \cap Q_{l(s)}(A)) &\ge d(\tau_{s} ,P_{s}) - 2^{|\tau_{s}|} \mu( \overline{Q_{l(s)}(A)})\\
&\ge f(s) - 2^{-4s-2} \\
&= {f(s)}/{2}.
\end{align*}

Hence if we let $\hat{P}_s = P_s \cap Q_{l(s)}(A)[m(s)]$, then the above inequality establishes that $d(\tau_s, \hat{P}_s) \ge f(s)/2$.
Let $t \in \omega$ be the least number such that there exists $\rho \succ \tau_s$ with the following properties: 
\begin{enumerate}
\item $\rho \not \in Q_{l(s)}(A)[t]$.
\item $d(\rho, \hat{P}_s) \ge f(s)/4$.
\end{enumerate}
Such a $t$ and $\rho$ exist by Lemma~\ref{lem: extension}. Let $\tau_{s+1}$ be the least $\rho$ for which the above holds for this particular $t$.

Now we will define $P_{s+1}$ such that $d(\tau_{s+1}, P_{s+1}) \ge f(s)/16$.
The definition of $P_{s+1}$ depends on the requirement being met.
If $s+1$ is odd, then let $e =s/2$. Using a $\zj$ oracle, we can find a number $n$ such that the measure of the set $\{X \in 2^\omega \colon \Gamma_e^X \succ D\until n\}$ is less than $2^{-|\tau_{s+1}|} \cdot f(s)/8$. The existence of such an $n$ follows from Sacks's theorem that the measure of the Turing cone above any non-computable set is zero \cite{Sacks_1963}. The class $P_{s+1}$ is defined to be the intersection of $\hat{P}_s$, and the class of sets that do not compute an extension $D\until n$ via $\Gamma_e$.

If $s+1$ is even, then let $e = (s+1)/2$. Using $\zj$, determine whether 
$d(\tau_{s+1}, \{X \in \hat{P}_s \colon \Gamma^X_e(e)\uparrow\}) \ge f(s)/8$. If so,
 set $P_{s+1} = \{X \in \hat{P}_s \colon \Gamma^X_e(e)\uparrow\}$.
Otherwise we have that 
$d(\tau_{s+1} , \{X \in \hat{P}_s \colon \Gamma^X_e(e)\downarrow\}) \ge f(s)/8$. Determine a number $n$
such that $d(\tau_{s+1} , \{X \in \hat{P}_s \colon \Gamma^X_e(e)\downarrow[n]\}) \ge f(s)/16$, and then set
$P_{s+1} = \{X \in \hat{P}_s \colon \Gamma^X_e(e)\downarrow[n]\}$.

Observe that $d(\tau_{s+1}, P_{s+1}) \ge f(s)/16 = f(s+1)$ hence our construction assumption holds for the following stage. This ends the construction. 

\emph{Verification.} The construction ensures that $R\not \ge_T D$ and $R' \le_T \zj$. We will show that 
$R\oplus A \equiv_T \zj$. Given $R$ and $A$, it possible to compute the sequence $\{\tau_s\}_{s\in\omega}$ and the function $l(s)$. Firstly, $l(s)$ can be computed from $|\tau_s|$. Secondly, given $l(s)$, we can compute the least stage $t_s$ such that for some $\tau \prec R$, $\tau \not \in Q_{l(s)}(A)[t_s]$. The string $\tau_{s+1}$ is the least such $\tau$ for this $t_s$. For all $s$, because $R \in Q_{l(s)}(A)[m(s)]$, $t_s$ is greater than $m(s)$, and thus $R\oplus A \ge_T \zj$.
\end{proof}

\begin{corollary}
If $A$ is a set below $\zj$, then $A$ is ML-cuppable if and only if $A$ is not $K$-trivial.
\end{corollary}

If we remove the requirements that $A$ and $D$ are below $\zj$, then the construction is computable in $A\oplus D \oplus \zj$ and we get a Martin-L\"of random $R$ such that $R \not \ge_T D$ and 
$A \oplus \zj \le_T A \oplus R \le_T A \oplus D \oplus \zj$. Letting $D=\zj$, we obtain the following corollary.

\begin{corollary}
If $A\subseteq\omega$, then $A$ is weakly ML-cuppable if and only if $A$ is not $K$-trivial.
\end{corollary}

Using a slightly different construction, we can construct, for any non-$K$-trivial $A$ and non-computable $D$, a Martin-L\"of random $R\ngeq_T D$ such that $A\oplus R \equiv_T A\oplus D \oplus \zj$. In order to achieve this, we need a new technique to encode $D$ into $A\oplus R$ while ensuring that $R\ngeq_T D$. First we need to improve on Lemma~\ref{lem: extension}.

\begin{lemma}
\label{lem: seq}
Let $P$ be a $\Pi^0_1$ class, $\tau \in 2^{<\omega}$, and $A$ a set that is not $K$-trivial. Let $Q$ be any $\Pi^0_1(A)$ class that only contains $A$-random sets. 
If $d(\tau, P \cap Q) \ge \delta$, then 
there exists an $A\oplus \zj$ computable prefix-free sequence of strings $\la \rho_i \colon i \in \omega \ra$ such that for all $i$,
\begin{enumerate}
\item \label{c l1} $\rho_i \not \in Q$. 
\item \label{c l2} $d(\rho_i, P) \ge \delta/2$.
\item \label{c l3} For all $\sigma\prec \rho_i$, there is an $s\in\omega$ with $\sigma \in Q[s]$ and $\rho_i \not\in Q[s]$.
\end{enumerate}
\end{lemma}
\begin{proof}
It is sufficient to show that $A\oplus \zj$ can compute a sequence of strings $\la \rho_i \colon i \in \omega \ra$, each with the properties \eqref{c l1}, \eqref{c l2} and \eqref{c l3}, such that for all $i, j \in \omega$ if $i < j$ then $\rho_i \not \preceq \rho_j$. 
From such a sequence, an infinite prefix-free set can be formed by removing any string that is a proper initial segment of an earlier string in the sequence.

We define $\la \rho_i \colon i \in \omega \ra$ by induction. Let $\rho_0$ be the first string found with properties \eqref{c l1}, \eqref{c l2} and \eqref{c l3}. Such a string exists by Lemma~\ref{lem: extension}. Once $\rho_s$ has been defined, let $P_s = P\setminus\bigcup_{i \le s} [\rho_i]$. Observe that $P_s \cap Q = P \cap Q$, so we can again apply Lemma~\ref{lem: extension} to find another extension of $\tau$ with the desired properties. Let this extension be $\rho_{s+1}$.
\end{proof}

\begin{theorem}\label{thm: more cupping}
If $A$ is a set that is not $K$-trivial and $D$ is a set that is non-computable, then there exists a Martin-L\"of random set $R$ such that:
\begin{enumerate}
\item $A\oplus R \equiv_T A \oplus D \oplus \zj$.
\item $R \not \ge_T D$.
\item $R' \le_T A \oplus R$.
\end{enumerate}
\end{theorem}
\begin{proof}
If $D \le_T A \oplus \zj$ then the result follows from the proof of Theorem~\ref{thm: cupping}, so we will assume that 
$D \not \le_T A \oplus \zj$.

\emph{Construction.} Define $f:\omega\rightarrow \omega$ by $f(s)= 2^{-5s-1}$. 
At stage $0$, let $\tau_0 = \lambda$. Let $P_0$ be the complement of the first level of the universal Martin-L\"of test. 
At stage $s+1$, let $l(s) = |\tau_s| + 5s+2$. Define $\hat{P}_s = P_s \cap Q_{l(s)}(A)[m(s)]$. 
Provided that $d(\tau_s, P_s) \ge f(s)$, then by a simple comparison of set sizes, 
we have that $d(\tau_s, \hat{P}_s \cap Q_{l(s)}(A)) \ge f(s)/2$. 

If $s+1$ is odd, then let $e=(s-1)/2$. Apply Lemma~\ref{lem: seq} to obtain a prefix-free sequence of strings $\la \rho_i \colon i \in \omega \ra$ by taking $P$, $Q$, $\tau$ and $\delta$ to be $\hat{P}_s$, $Q_{l(s)}(A)$, $\tau_s$ and $f(s)/2$ respectively. 
We claim that there is a $k$ such that 
\[
d(\rho_{\la D(e),k\ra}, \{X\in \hat{P}_s \colon \Gamma_e^X(k) = D(k)\}) \le \frac{1}{2}\; d(\rho_{\la D(e),k\ra},\hat{P}_s).
\]
This claim holds because otherwise $D$ could be computed from $A\oplus \zj$ by majority voting within $\hat{P}_s$. Specifically, for any $k$, determine the string $\rho_{\la D(e),k\ra}$, then using $\zj$ determine for which value $i \in \{0,1\}$, $\mu \{ X\in \hat{P}_s \cap [\rho_{\la D(e),k\ra}] \colon \Gamma_e^X(k)=i \} > \frac{1}{2} \mu (\hat{P}_s \cap [\rho_{\la D(e),k\ra}])$. By assumption, $D(k)=i$.
 
From $A\oplus D \oplus \zj$ we can find a $k$ such that
\begin{equation}\label{eqn: usable}
d(\rho_{\la D(e),k\ra}, \{X\in \hat{P}_s \colon \Gamma_e^X(k) = D(k)\}) < \frac{9}{16}\; d(\rho_{\la D(e),k\ra},\hat{P}_s).
\end{equation}
Let $\tau_{s+1} =\rho_{\la D(e), k\ra}$ for this $k$. This string $\tau_{s+1}$ allows us to encode the value of $D(e)$, and at the same time meet the requirement that $\Gamma^R_e \ne D$ by a judicious choice of $P_{s+1}$. 
Define $P_{s+1}$ as follows. Let $F= \{X \in2^\omega: \Gamma_e^X(k) \downarrow = 1 \}$,
and $G = \{X \in 2^\omega: \Gamma_e^X(k) \downarrow = 0 \}$ be $\Sigma^0_1$ classes. If any of the following conditions apply, then define $P_{s+1}$ as per the first condition that is found to hold by a $\emptyset'$ search. 
\begin{enumerate}
\item \label{l1} If $\displaystyle d(\tau_{s+1}, \hat{P}_s \cap(F \cup G)) < \frac{7}{8}\; d(\tau_{s+1}, \hat{P}_s)$, then let 
$P_{s+1} = \hat{P}_s \setminus(F \cup G)$. 
\item \label{l2} If $\displaystyle d(\tau_{s+1}, \hat{P}_s \cap F) > \frac{9}{16}\; d(\tau_{s+1}, \hat{P}_s)$, then $D(k)=0$ by \eqref{eqn: usable}, and so define 
$P_{s+1} = \hat{P}_s \setminus G$.
\item \label{l3} If $\displaystyle d(\tau_{s+1}, \hat{P}_s \cap G) > \frac{9}{16}\; d(\tau_{s+1}, \hat{P}_s)$, then $D(k)=1$ by \eqref{eqn: usable}, and so define 
$P_{s+1} = \hat{P}_s \setminus F$.
\item \label{l4} If $\displaystyle d(\tau_{s+1}, \hat{P}_s \cap F) > \frac{1}{4}\; d(\tau_{s+1}, \hat{P}_s)$ and $\displaystyle d(\tau_{s+1}, \hat{P}_s \cap G) > \frac{1}{4}\; d(\tau_{s+1}, \hat{P}_s)$, then if $D(k)=0$, let $P_{s+1} = \hat{P}_s \cap F[n]$ where $n$ is the least number such that $d(\tau_{s+1} ,\hat{P}_s \cap F[n]) \ge d(\tau_{s+1}, \hat{P}_s)/8$. If $D(k)=1$, then do the same with $F$ replaced by $G$.
\end{enumerate}
Note that one of these conditions must hold. If \eqref{l1} and \eqref{l2} both fail, then 
\[ d(\tau_{s+1}, \hat{P}_s \cap G) \geq \left(\frac{7}{8}-\frac{9}{16}\right)\; d(\tau_{s+1}, \hat{P}_s) > \frac{1}{4}\; d(\tau_{s+1}, \hat{P}_s),\]
 so the second conjunct in \eqref{l4} holds. If \eqref{l1} and \eqref{l3} both fail, we get the first conjunct in the same way. Also note that however we have defined $P_{s+1}$, we have 
 \[d(\tau_{s+1}, P_{s+1}) \geq d(\tau_{s+1}, \hat{P}_s)/8 \geq d(\tau_{s}, \hat{P}_s)/16 \geq d(\tau_s, P_s)/32 \geq f(s+1).\]

If $s+1$ is even, then act as in the construction of Theorem~\ref{thm: cupping}.

\emph{Verification.} The odd stages in the construction ensure that $R \not \ge_T D$. At stage $s+1=2e+1$, whichever condition \eqref{l1}--\eqref{l4} is used to define $P_{s+1}$ ensures that there is some $k$ such that no element of $P_{s+1}$ correctly computes $D(k)$ via $\Gamma_e$.
It is still the case that $A\oplus R \ge_T \zj$ by the same argument given in the proof of Theorem~\ref{thm: cupping}. 
We will show that using $A\oplus R$ we can determine indices for the $\Pi^0_1$ classes used in the construction. 
Note that once we have determined $\tau_{s+1}$, then $\emptyset'$ can determine which condition \eqref{l1}--\eqref{l4} is used to define $P_{s+1}$. We only need $D$ in this step if \eqref{l4} is used. In this case, $P_{s+1}$ is defined to be either $ \hat{P}_s \cap F[n]$ or $\hat{P}_s \cap G[n]$. But this means that $\Gamma_e^R(k)\downarrow$, and $F$ is used in the construction if and only if $\Gamma_e^R(k)\downarrow=1$. The $n$ is computable from $\zj$. This allows us to compute an index for $P_{s+1}$. 

From the index for $P_s$, along with $A$ and $R$, we can determine the index of the string $\rho_{\la i, k\ra}$ that is equal to $\tau_{s+1}$ when $s+1$ is odd. Hence we know that $D(e)=i$ for $e= (s-1)/2$. Finally, the even stages of the construction ensure that $R'$ is computable from the construction. Hence $R' \le_T A\oplus D \oplus \zj$, which we have just shown is Turing equivalent to $A\oplus R$.
\end{proof}

Theorem~\ref{thm: more cupping} lets us characterize $K$-triviality in terms of its degree-theoretic interaction with Martin-L\"of random sequences \emph{without} mentioning $\zj$.

\begin{corollary}
If $A\subseteq\omega$, then $A$ is not $K$-trivial if and only if for all $D>_T\emptyset$, there is a Martin-L\"of random $R\ngeq_T D$ such that and $A\oplus R\geq_T D$.
\end{corollary}
\begin{proof}
If $A$ is not $K$-trivial and $D>_T\emptyset$, then Theorem~\ref{thm: more cupping} gives us the necessary Martin-L\"of random. On the other hand, if $A$ is $K$-trivial, then let $D=\zj$ and use the fact that $A$ is not weakly ML-cuppable.
\end{proof}

%

Slaman and Steel extended the Posner-Robinson theorem to show that any non-computable set $A$ that is strictly Turing below $\zj$ can be cupped to $\zj$ with a 1-generic set $X$ such that $A$ and $X$ form a minimal pair \cite{Slam_Stee_1989}. The analogous result for $A$ not $K$-trivial and $X$ Martin-L\"of random does not hold. 
Any Martin-L\"of random computes a diagonally non-computable function. Ku\v{c}era showed that if $A$ and $B$ both compute diagonally non-computable functions and are both below $\zj$, then $A$ and $B$ do not form a minimal pair \cite{Kuce_1986}.
 Hence no Martin-L\"of random set below $\zj$ forms a minimal pair with any set $A$ below $\zj$ that computes a diagonally non-computable function.

The problem with adding minimal pair requirements to the construction used in the proof of Theorem~\ref{thm: cupping} is that $\zj$ cannot enumerate the non-computable, $A$-computable sets. However, $A''$ can and hence we can obtain the following.

\begin{corollary}
If $A$ is a set that is not $K$-trivial and $X \ge_T A''$, then there exists an incomplete Martin-L\"of random set $R$ such that $A\oplus R \equiv_T X$ and $A$ and $R$ form a minimal pair. 
\end{corollary}

\bibliographystyle{plain}
\bibliography{bibliography}

\end{document}